\documentclass[12pt]{article}
\usepackage{amssymb}
\usepackage{amsmath}
\usepackage{amsthm}
\usepackage{color}
\usepackage{comment}
\usepackage{geometry}		
\usepackage{graphicx}
\usepackage[table]{xcolor}   

\renewcommand{\labelenumi}{\roman{enumi})}

\DeclareFontFamily{OT1}{pzc}{}
\DeclareFontShape{OT1}{pzc}{m}{it}{<-> s * [1.10] pzcmi7t}{}
\DeclareMathAlphabet{\mathpzc}{OT1}{pzc}{m}{it}

\let\originalleft\left
\let\originalright\right
\renewcommand{\left}{\mathopen{}\mathclose\bgroup\originalleft}
\renewcommand{\right}{\aftergroup\egroup\originalright}

\newcommand\headercell[1]{%
   \smash[b]{\begin{tabular}[t]{@{}c@{}} #1 \end{tabular}}}

\begin{document}

\newcommand{\co}{\mathpzc{o}}
\newcommand\cF{\mathcal{F}}
\newcommand\cH{\mathcal{H}}
\newcommand\cX{\mathcal{X}}
\newcommand\id{{\rm id}}
\newcommand\rD{{\rm D}}
\newcommand\ee{\varepsilon}
\newcommand\modNormal{~\text{mod}~}
\newcommand\modSmall{\,\text{mod}\,}

\newtheorem{theorem}{Theorem}[section]
\newtheorem{corollary}[theorem]{Corollary}
\newtheorem{lemma}[theorem]{Lemma}
\newtheorem{proposition}[theorem]{Proposition}
\newtheorem{conjecture}[theorem]{Conjecture}

\theoremstyle{definition}
\newtheorem{definition}{Definition}[section]
\newtheorem{example}[definition]{Example}

\theoremstyle{remark}
\newtheorem{remark}{Remark}[section]




\title{
Border-collision bifurcations from stable fixed points to any number of coexisting chaotic attractors.
}
\author{
D.J.W.~Simpson\\\\
School of Mathematical and Computational Sciences\\
Massey University\\
Palmerston North, 4410\\
New Zealand
}
\maketitle


\begin{abstract}

In diverse physical systems stable oscillatory solutions devolve into more complicated dynamical behaviour through border-collision bifurcations. Mathematically these occur when a stable fixed point of a piecewise-smooth map collides with a switching manifold as parameters are varied. The purpose of this paper is to highlight the extreme complexity possible in the subsequent dynamics. We perturb instances of the border-collision normal form in $n \ge 2$ dimensions for which the $n^{\rm th}$ iterate is a direct product of identical skew tent maps that have chaotic attractors comprised of $k \ge 2$ disjoint intervals. The resulting maps have coexisting attractors and we use Burnside's lemma to count the number of mutually disjoint trapping regions produced by taking unions of Cartesian products of slight enlargements of the disjoint intervals. The attractors are shown to be chaotic by demonstrating that some iterate of the map is piecewise-expanding. The resulting transition from a stable fixed point to many coexisting chaotic attractors is shown to occur throughout open subsets of parameter space and not destroyed by adding higher order terms to the normal form, hence can be expected to arise generically in mathematical models.

\end{abstract}

\section{Introduction}
\label{sec:intro}

Piecewise-smooth maps
have different functional forms in different parts of phase space.
As a parameter of a piecewise-smooth map is varied,
a bifurcation occurs when a fixed point collides with a switching manifold, where the functional form of the map changes.
This type of bifurcation is termed a {\em border-collision bifurcation} (BCB).

BCBs have been identified in diverse applications.
Classical examples include power converters, where BCBs
can cause the internal dynamics of the converter to suddenly become quasi-periodic or chaotic \cite{ZhMo03,ZhMo06b},
and mechanical systems with friction, where BCBs
can induce recurring transitions between sticking and slipping motion \cite{DiBu08,SzOs09}.
More recently in \cite{RoHi19b}
it is shown how BCBs
can explain changes to the frequency of annual influenza outbreaks.

After the popularisation of BCBs by Nusse and Yorke in \cite{NuYo92},
it was quickly realised that BCBs often represent remarkably complicated transitions,
equivalent to the amalgamation of several (even infinitely many) smooth bifurcations \cite{NuYo95,BaGr99,Ho04,LeNi04}.
It is natural to then ask, how complicated can the transition be?
At a BCB the local attractor of a system
can change from a stable fixed point to a higher period solution, a quasiperiodic solution, or a chaotic solution \cite{Si16}.
It can also split into several attractors.
These attractors are created simultaneously and grow out of a single point.
If the parameter that affects the bifurcation is varied dynamically,
then in the presence of arbitrarily small noise it cannot be known {\em a priori} which attractor the system will transition to,
although the size of the basin of attraction of an attractor
can be expected to correlate positively with the likelihood that it will be selected \cite{DuNu99,KaMa98}.

Examples of BCBs creating multiple attractors
have been described by many authors \cite{BaGr99,SiMe08b,SuGa08,ZhMo08}.
Examples of BCBs creating arbitrarily many or infinitely many attractors was described in \cite{DoLa08,Si14,Si14b}.
The first numerical example of a BCB creating multiple chaotic attractors
is possibly that given in Section 7 of Avrutin {\em et.~al.}~\cite{AvSc12}.
More recently Pumari\~{n}o {\em et.~al.}~\cite{PuRo18} showed that two-dimensional, piecewise-linear maps
can exhibit $2^m$ coexisting chaotic attractors for any $m \ge 1$,
and under a coordinate transformation their maps are equivalent to members of the border-collision normal form.

The purpose of this paper is to demonstrate these complexities further.
Our approach extends that of Glendinning \cite{Gl15b}
who studied perturbations of the $n$-dimensional border-collision normal form
for which the $n^{\rm th}$ iterate is a direct product of identical skew tent maps.
But whereas Glendinning \cite{Gl15b} used skew tent maps that have a chaotic attractor consisting of one interval,
here we use skew tent maps that have a chaotic attractor consisting of $k$ disjoint intervals.
This novelty generates multiple attractors.
A similar strategy was employed by Wong and Yang \cite{WoYa19}
to get two chaotic attractors in the two-dimensional case.

Unlike previous works we take the extra step of proving that the bifurcation phenomenon is not an
artifact of the piecewise-linear nature of the border-collision normal form.
We show the phenomenon
persists when nonlinear terms are added to the normal form.
For a generic piecewise-smooth map (representing a mathematical model),
near a BCB the map is conjugate to a member of the normal form plus nonlinear terms \cite{Si16}.

To prove chaos we show some iterate of the map is piecewise-smooth and expanding.
Immediately this implies every Lyapunov exponent is positive,
but stronger results have been obtained in the context of ergodic theory.
Piecewise-$C^2$ expanding maps generically (i.e.~on an open, dense subset within the space of all such maps)
have at least one invariant measure that is absolutely continuous with respect to the Lebesgue measure \cite{Co02}.
The requirement that each piece is $C^2$ is not satisfied for BCBs
that correspond to grazing-sliding bifurcations (where the quadratic tangency of the grazing trajectory of the underlying system of differential equations
induces an order-$\frac{3}{2}$ error term).
In this case one can look to \cite{Sa00} for more general results regarding invariant measures.
In the two-dimensional case, if the map is piecewise-analytic
the genericity condition is not needed \cite{Bu00,Ts00}.
Analogous results for piecewise-linear maps are described in \cite{Bu99,Ts01}.

The remainder of the paper is organised as follows.
First in \S\ref{sec:results} we formally state the main results:
Theorem \ref{th:2d} for two-dimensional maps
and Theorem \ref{th:nd} for $n$-dimensional maps, for any $n \ge 2$.
In \S\ref{sec:perturb} we introduce a simple form for $n$-dimensional maps about which perturbations will be performed,
and show how nonlinear terms can be accommodated.
In \S\ref{sec:stableFixedPoint} we perform the straight-forward task of demonstrating that
our class of perturbed maps have an asymptotically stable fixed point on one side of the BCB.
Then in \S\ref{sec:directProduct} we show that on the other side of the BCB
the $n^{\rm th}$ iterate of the simple form is a direct product of identical skew tent maps.

In \S\ref{sec:skewTentMaps} we review attractors of skew tent maps
focussing on chaotic attractors that are comprised of $k \ge 2$ disjoint intervals.
In \S\ref{sec:trappingRegion} we fatten these intervals to obtain trapping regions for the skew tent maps.
Then in \S\ref{sec:symbolics} we take Cartesian products of intervals to obtain boxes,
where unions of the boxes form trapping regions for the $n$-dimensional maps.
In \S\ref{sec:number} we use Burnside's lemma and combinatorical arguments to
derive an explicit formula for the number of trapping regions that result from this construction as a function of $n$ and $k$.
We believe that each trapping region contains a unique attractor for sufficiently small perturbations,
but a proof of this remains for future work.
In the case $n=2$ we can obtain any number of trapping regions.

Then in \S\ref{sec:expansion} we prove that some iterate of the map is piecewise-smooth and expanding,
and in \S\ref{sec:proof} collate the results to prove Theorems \ref{th:2d} and \ref{th:nd}.
Lastly \S\ref{sec:conc} provides some final remarks.

\section{Main results}
\label{sec:results}

In this section we state our main results on the genericity of BCBs
where a stable fixed point bifurcates into several chaotic attractors.
Throughout the paper ${\rm int}(\cdot)$ denotes the interior of a set.

\begin{definition}
Let $f : D \to D$ be a map, where $D \subset \mathbb{R}^n$, and let $\Omega \subset D$ be compact.
If $f(\Omega) \subset {\rm int}(\Omega)$ then $\Omega$ is said to be a {\em trapping region} for $f$.
\label{df:trap}
\end{definition}

For any trapping region $\Omega$, the set $\bigcap_{i \ge 0} f^i(\Omega)$ is an {\em attracting set}, by definition \cite{Ro04}.
A {\em topological attractor} is then a subset of an attracting set that is dynamically indivisible, in some sense.
Different authors use different definitions for this indivisibility constraint (e.g.~there exists a dense orbit) \cite{Me07}.
For this paper we just require the fact that any trapping region contains at least one attractor.

To motivate the statement of the next definition,
note that in Definition \ref{df:trap} the restriction of $f$ to $\Omega$
is a map $f : \Omega \to \Omega$.

\begin{definition}
Let $\Omega \subset \mathbb{R}^n$ be compact.
A map $f : \Omega \to \Omega$ is said to be {\em piecewise-$C^r$} ($r \ge 1$)
if there exist finitely many mutually disjoint open regions $\Omega_i \subset \Omega$ such that
\begin{enumerate}
\renewcommand{\labelenumi}{\roman{enumi})}
\setlength{\itemsep}{0pt}
\item
each $\Omega_i$ has a piecewise-$C^r$ boundary,
\item
$\Omega$ is the union of the closures of the $\Omega_i$, and
\item
for each $i$ the map $f$ is $C^r$ on $\Omega_i$ and
can be extended so that it is $C^r$ on a neighbourhood of the closure of $\Omega_i$.
\end{enumerate}
\label{df:pws}
\end{definition}

\begin{definition}
The map $f$ of Definition \ref{df:pws} is {\em expanding} if there exists $\lambda > 1$ such that
\begin{equation}
\| \rD f(x) v \| \ge \lambda \| v \|,
\nonumber
\end{equation}
for all $x \in \bigcup_i \Omega_i$ and all $v \in \mathbb{R}^n$.
\label{df:expanding}
\end{definition}

The dynamics near any generic BCB in two dimensions can be described by a map of the form
\begin{equation}
x \mapsto \begin{cases}
\begin{bmatrix} \tau_L x_1 + x_2 + \mu \\ -\delta_L x_1 \end{bmatrix} + E_L(x;\mu), & x_1 \le 0, \\
\begin{bmatrix} \tau_R x_1 + x_2 + \mu \\ -\delta_R x_1 \end{bmatrix} + E_R(x;\mu), & x_1 \ge 0,
\end{cases}
\label{eq:2dbcnfWithHOTs}
\end{equation}
where $x = (x_1,x_2)$ is the state variable,
$\mu \in \mathbb{R}$ is the bifurcation parameter,
and $E_L$ and $E_R$ are $C^1$ and $\co \left( \| x \| + |\mu| \right)$ \cite{Si16}.
The last condition means $\frac{\| E_L(x;\mu) \|}{\| x \| + |\mu|} \to 0$ as $(x;\mu) \to (0;0)$
(regardless of how the limit is taken), and similarly for $E_R$.
Since BCBs are local, only one switching condition is relevant
and coordinates have been chosen so that it is $x_1 = 0$.
The BCB of \eqref{eq:2dbcnfWithHOTs} occurs at the origin $x=0$ when $\mu = 0$.
The nonlinear terms $E_L$ and $E_R$ often have no qualitative effect on the bifurcation
(as shown below for our setting)
and by dropping these terms we obtain a piecewise-linear family of maps known as the two-dimensional border-collision normal form.
Notice this form has four parameters $\tau_L, \delta_L, \tau_R, \delta_R \in \mathbb{R}$, in addition to $\mu$.

\begin{theorem}
For all $N \ge 1$ there exists an open set $U \subset \mathbb{R}^4$
such that for any piecewise-$C^r$ ($r \ge 1$) map $f$ of the form \eqref{eq:2dbcnfWithHOTs}
with $(\tau_L,\delta_L,\tau_R,\delta_R) \in U$,
there exists $\mu_0 > 0$ and $m \ge 1$ such that
\begin{enumerate}
\renewcommand{\labelenumi}{\roman{enumi})}
\setlength{\itemsep}{0pt}
\item
for all $\mu \in (-\mu_0,0)$, $f$ has an asymptotically stable fixed point, and
\item
for all $\mu \in (0,\mu_0)$, $f$ has $N$ disjoint trapping regions
on which $f^m$ is piecewise-$C^r$ and expanding.
\end{enumerate}
\label{th:2d}
\end{theorem}

Theorem \ref{th:2d} is proved in \S\ref{sec:proof}.
Fig.~\ref{fig:exBifDiag} shows a typical example with $N=2$.
This figure is for \eqref{eq:2dbcnfWithHOTs} with
\begin{align}
\tau_L &= -0.02, &
\delta_L &= -0.62, &
\tau_R &= -0.02, &
\delta_R &= 3,
\label{eq:xiEx}
\end{align}
and
\begin{align}
E_L(x;\mu) &= \begin{bmatrix} -x_1^2 \\ 0 \end{bmatrix}, &
E_R(x;\mu) &= \begin{bmatrix} 0 \\ 0 \end{bmatrix}.
\label{eq:ELEREx}
\end{align}
The bifurcation diagram, panel (a), shows that as the value of $\mu$ passes through $0$,
a stable fixed point turns into two coexisting attractors (coloured orange and black).
Panel (b) shows these attractors in phase space for one value of $\mu$.
Each attractor appears to be two-dimensional, like in \cite{WoYa19},
but the orange attractor has three connected components
while the black attractor has six connected components.
Numerically we observe the orange attractor is destroyed in a crisis \cite{GrOt83} at $\mu \approx 0.01$
after which all forward orbits converge to the black attractor.
This example is explained further in \S\ref{sec:conc}.

\begin{figure}[h!]
\begin{center}
\setlength{\unitlength}{1cm}
\begin{picture}(12.5,5)
\put(0,0){\includegraphics[height=5cm]{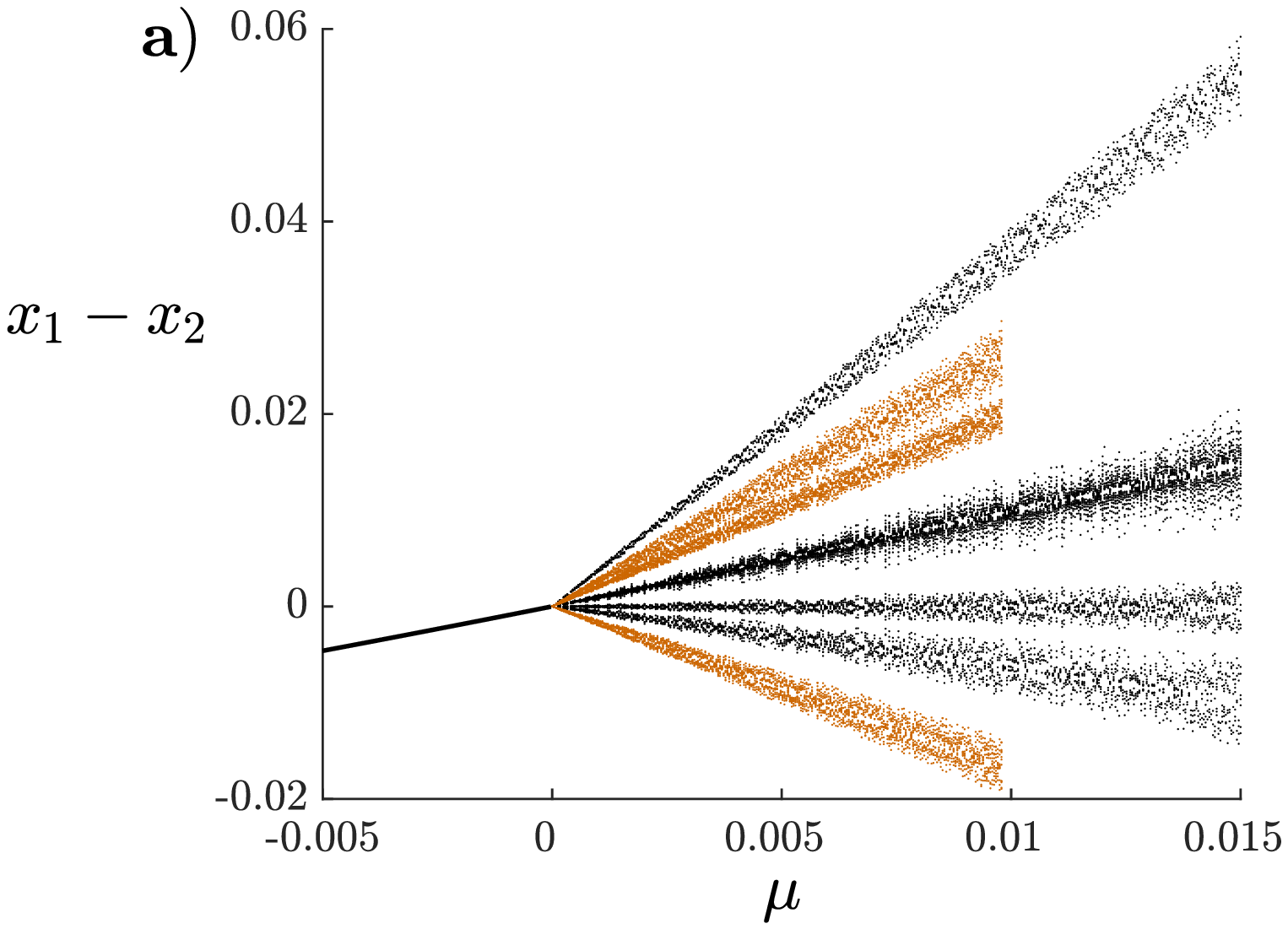}}
\put(7.5,0){\includegraphics[height=5cm]{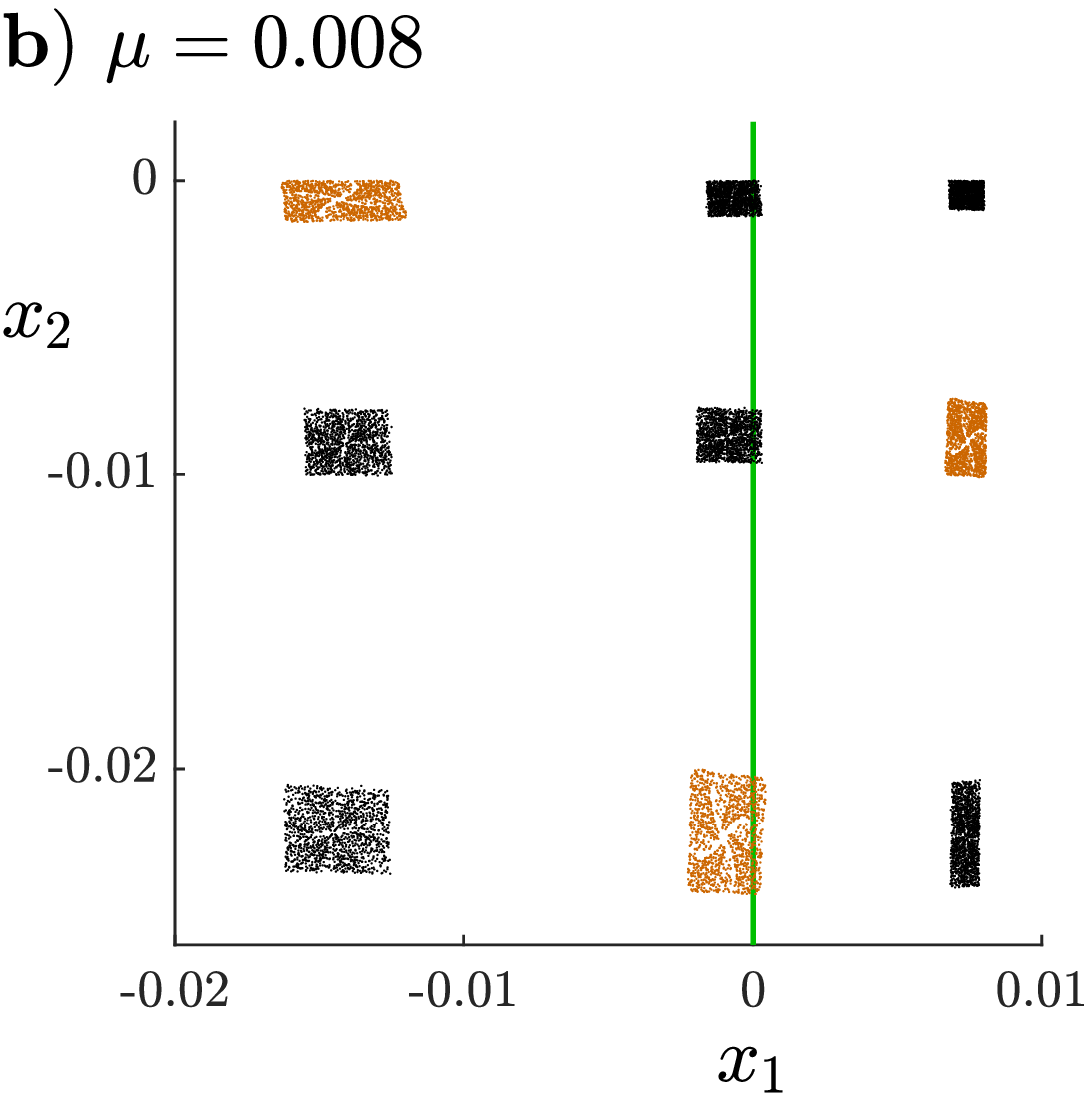}}
\end{picture}
\caption{
A numerically computed bifurcation diagram of \eqref{eq:2dbcnfWithHOTs} with
\eqref{eq:xiEx} and \eqref{eq:ELEREx}, and a phase portrait at $\mu = 0.008$.
The quantity plotted on the vertical axis of the bifurcation diagram
was chosen so that the two attractors can easily be distinguished.
\label{fig:exBifDiag}
} 
\end{center}
\end{figure}

We now consider BCBs in maps with more than two dimensions.
The form \eqref{eq:2dbcnfWithHOTs} generalises from $\mathbb{R}^2$ to $\mathbb{R}^n$ ($n \ge 2$) as
\begin{equation}
x \mapsto f(x;\mu) = \begin{cases}
C_L x + e_1 \mu + E_L(x;\mu), & x_1 \le 0, \\
C_R x + e_1 \mu + E_R(x;\mu), & x_1 \ge 0.
\end{cases}
\label{eq:ndbcnfWithHOTs}
\end{equation}
where again $E_L$ and $E_R$ are $C^1$ and $\co \left( \| x \| + |\mu| \right)$, and
\begin{align}
C_L &= \begin{bmatrix}
c^L_1 & 1 \\
c^L_2 && \ddots \\
\vdots &&& 1 \\
c^L_n
\end{bmatrix}, &
C_R &= \begin{bmatrix}
c^R_1 & 1 \\
c^R_2 && \ddots \\
\vdots &&& 1 \\
c^R_n
\end{bmatrix}
\label{eq:CLCR}
\end{align}
are companion matrices whose first columns are vectors $c^L, c^R \in \mathbb{R}^n$.
In \eqref{eq:ndbcnfWithHOTs} and throughout the paper we write
$e_j$ for the $j^{\rm th}$ standard basis vector of $\mathbb{R}^n$.
We now generalise Theorem \ref{th:2d} from $\mathbb{R}^2$ to $\mathbb{R}^n$.

\begin{theorem}
For all $k \ge 1$ there exists an open set $U \subset \mathbb{R}^n \times \mathbb{R}^n$
such that for any piecewise-$C^r$ ($r \ge 1$) map $f$ of the form \eqref{eq:ndbcnfWithHOTs} with $(c^L,c^R) \in U$,
there exists $\mu_0 > 0$ such that
\begin{enumerate}
\renewcommand{\labelenumi}{\roman{enumi})}
\item
for all $\mu \in (-\mu_0,0)$, $f$ has an asymptotically stable fixed point, and
\item
for all $\mu \in (0,\mu_0)$, $f$ has $N[k,n]$ disjoint trapping regions
on which $f^{k n}$ is piecewise-$C^r$ and expanding, where
\end{enumerate}
\begin{equation}
N[k,n] = \frac{1}{k n} \sum_{d|a} \varphi(d) k^{\frac{n}{d}},
\label{eq:Nformula}
\end{equation}
where $\varphi$ is Euler's totient function
and $a$ is the largest divisor of $n$ that is coprime to $k$.
\label{th:nd}
\end{theorem}

Theorem \ref{th:nd} is proved in \S\ref{sec:proof}.
In \eqref{eq:Nformula} the sum is over all divisors of $a$,
and $\varphi(d)$ is (by definition) the number of positive integers that are less than or equal to $d$ and coprime to $d$.
For example with $k = 10$ and $n = 6$, we have $a = 3$
whose divisors are $1$ and $3$.
Since $\varphi(1) = 1$ and $\varphi(3) = 2$, we have
\begin{equation}
N[10,6] = \frac{1}{60} \left( 1 \cdot 10^6 + 2 \cdot 10^2 \right) = 16670.
\nonumber
\end{equation}

\begin{table}[t!]
\caption{
Values of $N[k,n]$ \eqref{eq:Nformula}. 
\label{tb:N}
}
\begin{center}
\setlength\arrayrulewidth{.3mm}
\begin{tabular}{c|ccccc|}
\headercell{Number of\\bands, $k$} & \multicolumn{5}{c@{}}{Number of dimensions, $n$} \\
\cline{2-6}
\\[-4.72mm]  
& \cellcolor{gray!15}\parbox{12mm}{\centering 2}
& \cellcolor{gray!15}\parbox{12mm}{\centering 3}
& \cellcolor{gray!15}\parbox{12mm}{\centering 4}
& \cellcolor{gray!15}\parbox{12mm}{\centering 5}
& \cellcolor{gray!15}\parbox{12mm}{\centering 6} \\
\noalign{\hrule height .3mm}
\rule{0pt}{4mm}\cellcolor{gray!15}2 & 1 & 2 & 2 & 4 & 6 \\
\cellcolor{gray!15}3 & 2 & 3 & 8 & 17 & 42 \\
\cellcolor{gray!15}4 & 2 & 6 & 16 & 52 & 172 \\
\cellcolor{gray!15}5 & 3 & 9 & 33 & 125 & 527 \\
\cellcolor{gray!15}6 & 3 & 12 & 54 & 260 & 1296 \\
\cellcolor{gray!15}7 & 4 & 17 & 88 & 481 & 2812 \\
\cellcolor{gray!15}8 & 4 & 22 & 128 & 820 & 5464 \\
\cellcolor{gray!15}9 & 5 & 27 & 185 & 1313 & 9855 \\
\cellcolor{gray!15}10 & 5 & 34 & 250 & 2000 & 16670 \\
\noalign{\hrule height .3mm}
\end{tabular}
\end{center}
\end{table}

Table \ref{tb:N} lists the values of $N[k,n]$ for small values of $k$ and $n$.
If $k$ is a multiple of $n$, then $a=1$ and $N = \frac{k^{n-1}}{n}$. 
If $n$ is prime and $k$ is not a multiple of $n$, then $a = n$
has only two divisors: $1$ and $n$.
Here $\varphi(n) = n-1$ so $N = \frac{k^{n-1} + n - 1}{n}
= \left\lceil \frac{k^{n-1}}{n} \right\rceil$.
Interestingly this implies $k^{n-1} - 1$ is a multiple of $n$ which is Fermat's little theorem;
compare \cite{ChNo13,Go56}.
Also $N[k,2] = \left\lceil \frac{k}{2} \right\rceil$
and $N[2,n]$ is Sloane's integer sequence $A000016$ \cite{SlPl95}
which arises in coding problems \cite{Go17};
see \cite{St08} for its occurrence in another dynamical systems setting.

\section{Perturbations from a two-parameter family}
\label{sec:perturb}

In the absence of the nonlinear terms $E_L$ and $E_R$,
\eqref{eq:ndbcnfWithHOTs} is the $n$-dimensional border-collision normal form, first considered in \cite{Di03}.
In this section we introduce a two-parameter subfamily of the normal form about which perturbations will be taken.

Given parameters $a_L, a_R \in \mathbb{R}$ and $\sigma = \pm 1$, we form the $n$-dimensional map
\begin{equation}
g(y;a_L,a_R,\sigma) = \begin{cases}
A_L y + \sigma e_1 \,, & y_1 \le 0, \\
A_R y + \sigma e_1 \,, & y_1 \ge 0,
\end{cases}
\label{eq:simpleForm}
\end{equation}
where
\begin{align}
A_L &= \begin{bmatrix}
0 & 1 \\
\vdots && \ddots \\
0 &&& 1 \\
a_L
\end{bmatrix}, &
A_R &= \begin{bmatrix}
0 & 1 \\
\vdots && \ddots \\
0 &&& 1 \\
a_R
\end{bmatrix}.
\label{eq:ALAR}
\end{align}
The matrices $A_L$ and $A_R$ are companion matrices \eqref{eq:CLCR}
for which $c^L$ and $c^R$ are scalar multiples of $e_n$.
Specifically $c^L = d^L$ and $c^R = d^R$ where
\begin{align}
d^L &= a_L e_n \,, &
d^R &= a_R e_n \,.
\nonumber
\end{align}

Let us now explain the utility of \eqref{eq:simpleForm}.
To a map of the form \eqref{eq:ndbcnfWithHOTs}, we perform the spatial scaling $y = \frac{x}{|\mu|}$,
assuming $\mu \ne 0$, to produce the map
\begin{equation}
y \mapsto \frac{f(|\mu| y;\mu)}{|\mu|} = \tilde{f}(y;\mu).
\label{eq:scaledMapDefn}
\end{equation}
This map can be written as
\begin{equation}
\tilde{f}(y;\mu) = 
g(y;a_L,a_R,{\rm sgn}(\mu)) +
\begin{cases}
(c^L - d^L) y_1 + \frac{E_L(|\mu| y;\mu)}{|\mu|}, & y_1 \le 0, \\
(c^R - d^R) y_1 + \frac{E_R(|\mu| y;\mu)}{|\mu|}, & y_1 \ge 0.
\end{cases}
\label{eq:scaledMap}
\end{equation}
By choosing $\mu$ small and $c^L$ and $c^R$ close to $d^L$ and $d^R$,
we can make the difference between \eqref{eq:ndbcnfWithHOTs}, in scaled coordinates, as close to
the simple form \eqref{eq:simpleForm} as we like.
Formally we have the following result
that follows immediately from the assumption that $E_L$ and $E_R$
are $C^1$ and $\co \left( \| x \| + |\mu| \right)$.

\begin{lemma}
Let $\Omega \subset \mathbb{R}^n$ be bounded and $a_L, a_R \in \mathbb{R}$.
For all $\eta > 0$ there exists
a neighbourhood $U_1 \subset \mathbb{R}^n \times \mathbb{R}^n$ of $(d^L,d^R)$ such that
for any map $f$ of the form \eqref{eq:ndbcnfWithHOTs} with $(c^L,c^R) \in U_1$
there exists $\mu_1 > 0$ such that for all $\mu \in (\mu_1,0) \cup (0,\mu_1)$ and all $y \in \Omega$ we have
\begin{enumerate}
\renewcommand{\labelenumi}{\roman{enumi})}
\item
$\big\| \tilde{f}(y;\mu) - g(y;a_L,a_R,{\rm sgn}(\mu)) \big\| < \eta$,
\item
$\big\| \rD \tilde{f}(y;\mu) - A_L \big\| < \eta$ if $y_1 > 0$, and
\item
$\big\| \rD \tilde{f}(y;\mu) - A_R \big\| < \eta$ if $y_1 > 0$.
\end{enumerate}
\label{le:perturb}
\end{lemma}

\section{A stable fixed point for $\mu < 0$}
\label{sec:stableFixedPoint}

In this section we generalise Section 2 of Glendinning \cite{Gl15b} to accommodate higher order terms.

\begin{lemma}
Let $a_L, a_R \in \mathbb{R}$ with $|a_L| < 1$.
There exists a neighbourhood $U_2 \subset \mathbb{R}^n \times \mathbb{R}^n$ of $(d^L,d^R)$ such that
for any map $f$ of the form \eqref{eq:ndbcnfWithHOTs} with $(c^L,c^R) \in U_2$
there exists $\mu_2 > 0$ such that for all $\mu \in (-\mu_2,0)$
the map $f$ has an asymptotically stable fixed point.
\label{le:stableFixedPoint}
\end{lemma}

\begin{proof}
The left piece of $g(y;a_L,a_R,-1)$ has the unique fixed point
\begin{equation}
y^* = \frac{-1}{1 - a_L} \begin{bmatrix} 1 \\ a_L \\ a_L \\ \vdots \\ a_L \end{bmatrix}.
\label{eq:yStar}
\end{equation}
The first component of $y^*$ is negative (because $a_L < 1$) thus $y^*$ is {\em admissible},
i.e.~it is a fixed point of $g$.
The stability multipliers associated with $y^*$ are the eigenvalues of $A_L$.
Since $\det \left( \lambda I - A_L \right) = \lambda^n - a_L$,
these eigenvalues are the $n^{\rm th}$ roots of $a_L$,
so all have modulus less than $1$ (because $|a_L| < 1$).
Thus $y^*$ is asymptotically stable and hyperbolic,
so persists for maps that are a sufficiently small perturbation of \eqref{eq:simpleForm}.
Thus the result follows from Lemma \ref{le:perturb}.
\end{proof}

\section{A direct product of skew tent maps}
\label{sec:directProduct}

For the remainder of the paper we consider $g$ with $\sigma = 1$.
Here we show that the $n^{\rm th}$ iterate of $g$ is conjugate to a direct product of skew tent maps.
To this end, we let
\begin{equation}
h(z;a_L,a_R) = \begin{cases}
a_L z + 1, & z \le 0, \\
a_R z + 1, & z \ge 0,
\end{cases}
\label{eq:skewTentMap}
\end{equation}
be a family of skew tent maps.

\begin{lemma}
The $n^{\rm th}$ iterate of \eqref{eq:simpleForm} with $\sigma = 1$ is
\begin{equation}
g^n(x;a_L,a_R,1) = \begin{bmatrix}
h(x_1;a_L,a_R) \\
h(x_2+1;a_L,a_R)-1 \\
\vdots \\
h(x_n+1;a_L,a_R)-1 \\
\end{bmatrix}.
\label{eq:directProduct}
\end{equation}
\label{le:directProduct}
\end{lemma}

\begin{proof}
The last component of $g(x)$ is $a_L x_1$ if $x_1 \le 0$ and $a_R x_1$ otherwise.
That is $g(x)_n = h(x_1) - 1$.
By further iterating under $g$ we obtain
$h(x_1)-1 = g^2(x)_{n-1} = g^3(x)_{n-2} = \cdots = g^{n-1}(x)_2$.
Iterating one more time gives $g^n(x)_1 = h(x_1)$
which verifies the first component of \eqref{eq:directProduct}.
The remaining components can be verified similarly.
\end{proof}

\section{A divison of the parameter space of skew tent maps}
\label{sec:skewTentMaps}

A detailed analysis of the skew tent map family \eqref{eq:skewTentMap} was done in Ito {\em et.~al.}~\cite{ItTa79b},
see also Maistrenko {\em et.~al.}~\cite{MaMa93}.
For any $a_L, a_R \in \mathbb{R}$, \eqref{eq:skewTentMap}
has non-negative Schwarzian derivative almost everywhere so has at most one attractor \cite{CoEc80}.
Fig.~\ref{fig:zccbifSet1dBCNF} catagorises this attractor throughout the $(a_L,a_R)$-parameter plane.
It contains regions $P_k$, for $k \ge 1$, where there exists a stable period-$k$ solution.
It also contains regions $Q_k$, for $k = 1$ and $k = 2^\ell$ for $\ell \ge 3$,
regions $R_k$, for even $k \ge 4$, and regions $S_k$, for $k \ge 2$,
where there exists a chaotic attractor consisting of $k$ disjoint closed intervals.

As evident in Fig.~\ref{fig:zccbifSet1dBCNF}, each $P_k$ is situated below $P_{k-1}$.
For all $k \ge 3$, $R_{2k}$ is narrow and located immediately to the right of $P_k$,
while $S_k$ is similarly narrow and located immediately to the right of $R_{2k}$.

In this paper we perturb about instances of \eqref{eq:simpleForm}
for which the pair $(a_L,a_R)$ belongs to some region $S_k$.
In comparison Wong and Yang \cite{WoYa19}
use $(a_L,a_R) = (0.5,-2)$, which lies on the boundary of $P_2$ and $R_4$,
while Glendinning allows any $\frac{6}{11} < a_L < \frac{10}{11}$ and $a_R = -2$,
which includes points in $R_4$, $S_2$, and $Q_1$.
All constructions use $|a_L| < 1$ to ensure a stable fixed point for small $\mu < 0$ by Lemma \ref{le:stableFixedPoint}.

\begin{figure}[h!]
\begin{center}
\includegraphics[height=9cm]{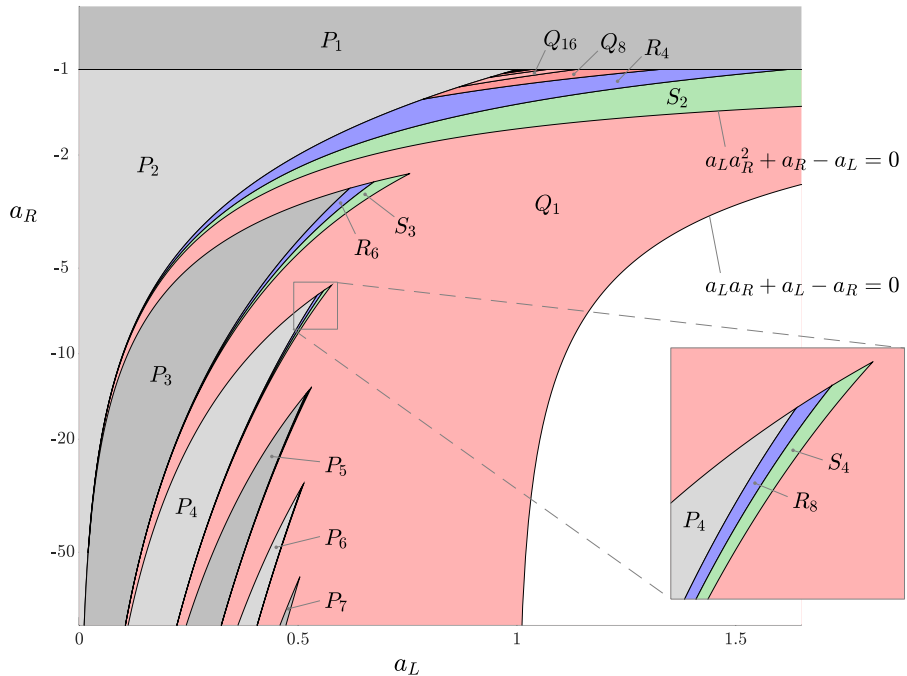}
\caption{
A two-parameter bifurcation diagram of the skew tent map family \eqref{eq:skewTentMap}.
The regions $P_k$ are where the map has a stable period-$k$ solution,
while regions $Q_k$, $R_k$, and $S_k$ are where the map has a chaotic attractor comprised of $k$ disjoint intervals.
\label{fig:zccbifSet1dBCNF}
} 
\end{center}
\end{figure}

Each $S_k$ is bounded by three smooth curves.
The upper boundary curve is where $h^k(0) = 0$ and is given by
\begin{equation}
a_R = -\frac{1 - a_L^{k-1}}{(1-a_L) a_L^{k-2}}.
\label{eq:upper}
\end{equation}
The left and right boundary curves of $S_k$ are given by
\begin{align}
a_L^{2 k - 2} a_R^3 + a_L - a_R &= 0, \label{eq:left} \\
a_L^{k - 1} a_R^2 + a_R - a_L &= 0, \label{eq:right}
\end{align}
respectively.
Some explanation for these is provided below.

\begin{lemma}
For any $(a_L,a_R) \in S_k$ ($k \ge 2$),
\begin{equation}
\begin{split}
&h^2(0) < h^{k+2}(0) < h^3(0) < h^{k+3}(0) < \cdots < h^{2k-1}(0) < h^k(0) < 0 \\
&\text{and} \quad 0 < h^{2k}(0) < h^{k+1}(0) < h^{2k+1}(0) < h(0).
\end{split}
\label{eq:ordering}
\end{equation}
\label{le:ordering}
\end{lemma}

\begin{figure}[h!]
\begin{center}
\includegraphics[height=9cm]{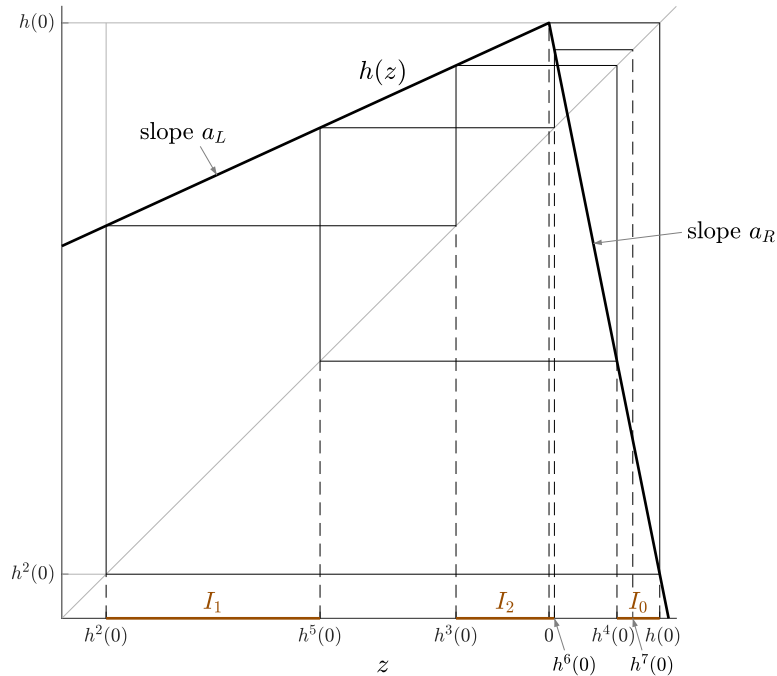}
\caption{
A cobweb diagram of the skew tent map \eqref{eq:skewTentMap} with $(a_L,a_R) \in S_3$.
The forward orbit of the origin $h^i(0)$ is indicated for all $0 \le i \le 7$.
This orbit is used to define intervals $I_0$, $I_1$, and $I_2$ by \eqref{eq:Ii} with $k = 3$.
\label{fig:zccCobwebSchem}
} 
\end{center}
\end{figure}

The ordering \eqref{eq:ordering} is illustrated for $k=3$ in Fig.~\ref{fig:zccCobwebSchem}.
Lemma \ref{le:ordering} is a consequence of calculations done in \cite{ItTa79b,MaMa93}
and, although a little tedious, is not difficult to derive directly from \eqref{eq:upper}--\eqref{eq:right}.
Note the ordering \eqref{eq:ordering} also holds throughout $R_{2 k}$.

We now use the forward orbit of $0$ to define intervals
\begin{equation}
\begin{split}
I_0 &= \left[ h^{k+1}(0), h(0) \right], \\
I_i &= \left[ h^{i+1}(0), h^{i+k+1}(0) \right], ~\text{for all $i = 1,\ldots,k-1$},
\end{split}
\label{eq:Ii}
\end{equation}
see again Fig.~\ref{fig:zccCobwebSchem}.
The next result follows immediately from \eqref{eq:ordering}.

\begin{lemma}
For any $(a_L,a_R) \in S_k$ ($k \ge 2$),
the intervals \eqref{eq:Ii} are mutually disjoint and
$h(I_i) = I_{i+1 \modSmall k}$ for all $i = 0,1,\ldots,k-1$.
\label{le:Ii}
\end{lemma}

By Lemma \ref{le:Ii} orbits cycle through the intervals $I_i$,
one of which ($I_{k-1}$) contains the critical point $z=0$.
Thus the restriction of $h^k$ to any of these intervals
is a continuous piecewise-linear map with two pieces,
see Fig.~\ref{fig:zccInducedSchem}.
That is, $h^k$ is conjugate to an instance of the skew tent map family \eqref{eq:skewTentMap}.
Specifically, $h^k$ is conjugate to $h(z;\tilde{a}_L,\tilde{a}_R)$, where
$\tilde{a}_L = a_L^{k-2} a_R^2$ and $\tilde{a}_R = a_L^{k-1} a_R$,
because in each cycle orbits undergo either one or two iterations under the right piece of $h$,
and the remaining iterations under the left piece of $h$.

\begin{figure}[h!]
\begin{center}
\includegraphics[height=6cm]{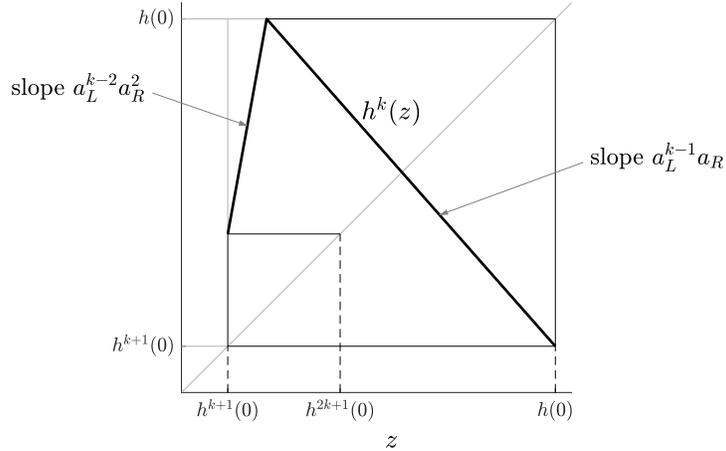}
\caption{
A sketch of the $k^{\rm th}$ iterate of \eqref{eq:skewTentMap} with $(a_L,a_R) \in S_k$
on the interval $I_0$.
\label{fig:zccInducedSchem}
} 
\end{center}
\end{figure}

If $(\tilde{a}_L,\tilde{a}_R) \in Q_1$ then $h(z;\tilde{a}_L,\tilde{a}_R)$ is transitive on $\tilde{I} = [h^2(0),h(0)]$.
Refer to \cite{VeGl90} (Lemmas 2.1 and 2.2) for a simple proof of this that assumes $\tilde{a}_L > 1$, as is the case here.
This implies $h(z;a_L,a_R)$ is transitive on $I_0 \cup \cdots \cup I_{k-1}$.
Since the value of $\tilde{a}_L$ is relatively large,
to be in $Q_1$ we just need $(\tilde{a}_L,\tilde{a}_R)$ to lie between
the two boundaries of $Q_1$ that are labelled in Fig.~\ref{fig:zccbifSet1dBCNF}.
This is because below $a_L a_R + a_L - a_R = 0$ the interval $\tilde{I}$ is not forward invariant,
while above $a_L a_R^2 + a_R - a_L = 0$ transitivity fails because
typical orbits cannot reach a neighbourhood of the fixed point in $\tilde{I}$.

By replacing $a_L$ and $a_R$ with $\tilde{a}_L$ and $\tilde{a}_R$
in the formulas for these boundaries
we produce \eqref{eq:left} and \eqref{eq:right}.
This explains the particular formulas \eqref{eq:left} and \eqref{eq:right}
and gives the following result.

\begin{lemma}
For any $(a_L,a_R) \in S_k$ ($k \ge 2$),
$h$ is transitive on $I_0 \cup \cdots \cup I_{k-1}$.
\label{le:transitive}
\end{lemma}

\section{Constructing fattened intervals}
\label{sec:trappingRegion}

We now fatten the intervals $I_i$ \eqref{eq:Ii} to form a trapping region for
the attractor $I_0 \cup \cdots \cup I_{k-1}$ of the one-dimensional map $h$.

\begin{lemma}
For any $(a_L,a_R) \in S_k$ ($k \ge 2$),
there exists $\delta > 0$ such that the intervals
\begin{equation}
\begin{split}
J_0 &= \left[ h^{k+1}(0) - (k+1) \delta a_L^{k-1} |a_R|, h(0) + \delta \right], \\
J_i &= \left[ h^{i+1}(0) - (i+1) \delta a_L^{i-1} |a_R|, h^{i+k+1}(0) + (i+k+1) \delta a_L^{k+i-2} a_R^2 \right], ~\text{for all $i = 1,\ldots,k-1$},
\end{split}
\label{eq:Ji}
\end{equation}
are mutually disjoint and
\begin{equation}
h(J_i) \subset {\rm int} \left( J_{i+1 \modSmall k} \right), ~\text{for all $i = 0,1,\ldots,k-1$}.
\label{eq:Jimap}
\end{equation}
\label{le:Jimap}
\end{lemma}

\begin{proof}
By Lemma \ref{le:Ii} we can take $\delta > 0$ sufficiently small that the intervals $J_i$ are mutually disjoint.
The interval $J_0$ maps under $a_R z + 1$, where $a_R < 0$, thus
\begin{equation}
h(J_0) = \left[ h^2(0) - |a_R| \delta, h^{k+2}(0) + (k+1) \delta a_L^{k-1} |a_R|^2 \right].
\nonumber
\end{equation}
By comparing this to
\begin{equation}
J_1 = \left[ h^2(0) - 2 |a_R| \delta, h^{k+2}(0) + (k+2) \delta a_L^{k-1} |a_R|^2 \right],
\nonumber
\end{equation}
we verify \eqref{eq:Jimap} for $i=0$.
Equation \eqref{eq:Jimap} can similarly be verified for each $i=1,\ldots,k-2$.

For $i = k-1$, observe $0 \in {\rm int}(J_{k-1})$, by \eqref{eq:ordering}.
Thus the image under $h$ of the left endpoint of $J_{k-1}$ is $h^{k+1}(0) - k \delta a_L^{k-1} |a_R|$,
while the image under $h$ of the right endpoint is $h^{2k+1}(0) + 2 k \delta a_L^{2 k - 3} |a_R|^3$.
Since $h^{k+1} < h^{2k+1}$, by \eqref{eq:ordering},
we can choose $\delta > 0$ small enough that the image of the left endpoint is smaller than the image of the right endpoint.
In this case
\begin{equation}
h(J_{k-1}) = \left[ h^{k+1}(0) - k \delta a_L^{k-1} |a_R|, h(0) \right],
\nonumber
\end{equation}
and by comparing this to the definition of $J_0$ we see that \eqref{eq:Jimap} is verified for $i = k-1$.
\end{proof}

To motivate the next construction, recall
from Lemma \ref{le:directProduct} that $g^n$ is conjugate (via a translation)
to a direct product of $n$ copies of $h$.
So by Lemma \ref{le:Jimap} to obtain trapping regions for $g^n$ we can take
unions of Cartesian products of the $J_i$, some shifted by $-1$ to account for the translation in \eqref{eq:directProduct}.
But to obtain trapping regions for $g$ we have to work a bit harder.

Write
\begin{align*}
J_i &= [p_i,q_i], \\
h \left( J_{i-1 \modSmall k} \right) &= [r_i,s_i],
\end{align*}
for all $i = 0,1,\ldots,k-1$.
Observe $p_i < r_i < s_i < q_i$ for all $i$ by \eqref{eq:Jimap}.
Now let
\begin{align}
t_{i,j} &= p_i + \frac{(r_i - p_i)(j-1)}{n}, &
u_{i,j} &= q_i + \frac{(s_i - q_i)(j-1)}{n},
\label{eq:tu}
\end{align}
and
\begin{equation}
K_{i,j} = \left[ t_{i,j} - 1, u_{i,j} - 1 \right],
\label{eq:K}
\end{equation}
for all $i = 0,1,\ldots,k-1$ and $j = 2,3,\ldots,n$.
The next result uses the following notation:
given $Z \subset \mathbb{R}$ and $a \in \mathbb{R}$,
we write $Z + a$ to abbreviate $\{ z+a \,|\, z \in Z \}$.
Equation \eqref{eq:JKmap} is an immediate consequence of the ordering 
illustrated in Fig.~\ref{fig:zccIntervals}.

\begin{figure}[h!]
\begin{center}
\includegraphics[height=1cm]{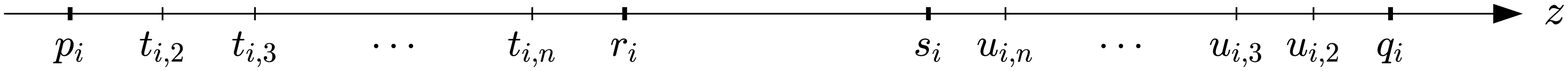}
\caption{
A sketch illustrating the ordering of scalar quantities introduced in the text.
These are defined from the fattened intervals $J_i$
and used to construct additional intervals $K_{i,j}$
that we then use to form boxes via \eqref{eq:Phi}.
Note, each $I_i$ is contained in $[r_i,s_i]$.
\label{fig:zccIntervals}
} 
\end{center}
\end{figure}

\begin{lemma}
Let $(a_L,a_R) \in S_k$ ($k \ge 2$) and $\delta > 0$ be as in Lemma \ref{le:Jimap}.
Then
\begin{equation}
\begin{split}
h \left( J_{i-1 \modSmall k} \right) - 1 &\subset {\rm int} \left( K_{i,n} \right), \\
K_{i,j} &\subset {\rm int} \left( K_{i,j-1} \right), ~\text{for all $j = 3,4,\ldots,n$}, \\
K_{i,2} + 1 &\subset {\rm int} \left( J_i \right),
\end{split}
\label{eq:JKmap}
\end{equation}
for all $i = 0,1,\ldots,k-1$.
\label{le:JKmap}
\end{lemma}

\section{Boxes}
\label{sec:symbolics}

We write $\mathbb{Z}_k$ for the set $\{ 0,1,\ldots,k-1 \}$ with addition taken modulo $k$.
Given a vector $v \in \mathbb{Z}_k^n$ (so $v_j \in \mathbb{Z}_k$ for each $j = 1,2,\ldots,n$), we let
\begin{equation}
\Phi_v = J_{v_1} \times K_{v_2,2} \times K_{v_3,3} \times \cdots \times K_{v_n,n} \,,
\label{eq:Phi}
\end{equation}
be a box in $\mathbb{R}^n$.
By Lemma \ref{le:JKmap} and the definition of $g$, each such box
maps under $g$ to the interior of another such box.
Specifically $\Phi_v$ maps to the interior of $\Phi_{\psi(v)}$, where the
map $\psi : \mathbb{Z}_k^n \to \mathbb{Z}_k^n$ is defined by
\begin{equation}
\psi(v) = \left( v_2, v_3, \ldots, v_n, v_1 + 1 \right).
\label{eq:psi}
\end{equation}
Formally we have the following result.

\begin{lemma}
Let $(a_L,a_R) \in S_k$ ($k \ge 2$) and $\delta > 0$ be as in Lemma \ref{le:Jimap}.
Then
\begin{equation}
g(\Phi_v) \subset {\rm int} \left( \Phi_{\psi(v)} \right),
\label{eq:boxImage}
\end{equation}
for all $v \in \mathbb{Z}_k^n$.
\label{le:boxImage}
\end{lemma}

We now use Lemma \ref{le:perturb} to extend this result to maps of the form \eqref{eq:ndbcnfWithHOTs}.
Here we use the following notation:
given $\Omega \subset \mathbb{R}^n$,
we write $\mu \Omega$ to abbreviate $\{ \mu y \,|\, y \in \Omega \}$.

\begin{lemma}
Let $(a_L,a_R) \in S_k$ ($k \ge 2$) and $\delta > 0$ be as in Lemma \ref{le:Jimap}.
There exists a neighbourhood $U_3 \subset \mathbb{R}^n \times \mathbb{R}^n$ of $(d^L,d^R)$ such that
for any map $f$ of the form \eqref{eq:ndbcnfWithHOTs} with $(c^L,c^R) \in U_3$
there exists $\mu_3 > 0$ such that
\begin{equation}
f(\mu \Phi_v;\mu) \subset {\rm int} \left( \mu \Phi_{\psi(v)} \right),
\label{eq:unionPhiMap}
\end{equation}
for all $\mu \in (0,\mu_3)$ and $v \in \mathbb{Z}_k^n$.
\label{le:boxImagePerturbed}
\end{lemma}

\begin{proof}
By \eqref{eq:boxImage} there exists $\eta > 0$ such that, for all $v \in \mathbb{Z}_k^n$,
all points within a distance $\eta$ of $g(\Phi_v)$ lie inside ${\rm int} \left( \Phi_{\psi(v)} \right)$.
With $\Omega = \bigcup_{v \in \mathbb{Z}_k^n} \Phi_v$, let $U_3 = U_1$ and $\mu_3 = \mu_1$ be as in Lemma \ref{le:perturb}.
Then for any map of the form \eqref{eq:ndbcnfWithHOTs} with $(c^L,c^R) \in U_3$,
\eqref{eq:unionPhiMap} is satisfied for all $v \in \mathbb{Z}_k^n$.
\end{proof}

\section{Counting the number of trapping regions}
\label{sec:number}

The {\em orbit} of $v \in \mathbb{Z}_k^n$ under $\psi$ is the set
\begin{equation}
{\rm orb}(v) = \left\{ \psi^i(v) \,\middle|\, i \ge 0 \right\}.
\label{eq:orb}
\end{equation}
Given $v \in \mathbb{Z}_k^n$ let
\begin{equation}
T_v = \bigcup_{w \in {\rm orb}(v)} \Phi_w \,.
\label{eq:T}
\end{equation}
Then the scaled set $\mu T_v$ is a trapping region for any map $f$
that satisfies the conditions of Lemma \ref{le:boxImagePerturbed}.
The number of mutually disjoint trapping regions
given by this construction is equal to the number of orbits of $\psi$.
The purpose of this section is to prove the following result.

\begin{proposition}
For any $k, n \ge 1$ the number of orbits of $\psi$ is given by \eqref{eq:Nformula}.
\label{pr:numberOrbits}
\end{proposition}

\begin{figure}[h!]
\begin{center}
\includegraphics[height=9cm]{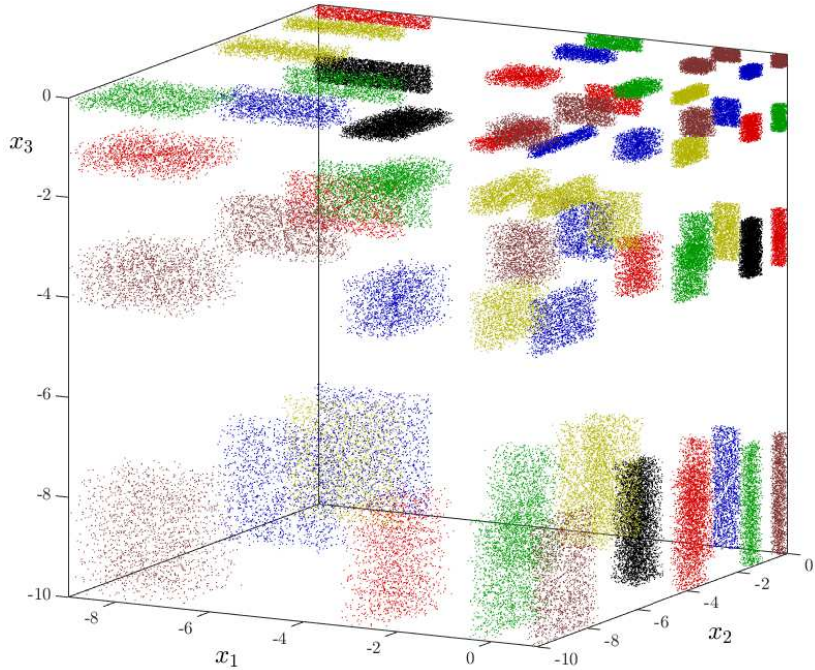}
\caption{
A phase portrait of \eqref{eq:simpleForm} with $n=3$ and $(a_L,a_R) = (0.47,-10) \in S_4$.
We show $20000$ points of six different orbits.
Each orbit has a different colour.
\label{fig:zccqq3d}
} 
\end{center}
\end{figure}

Fig.~\ref{fig:zccqq3d} shows an example with $k=4$ and $n=3$.
This figure is for the simple form \eqref{eq:simpleForm} in three dimensions with $(a_L,a_R) = (0.47,-10) \in S_4$.
By Proposition \ref{pr:numberOrbits} the number of trapping regions \eqref{eq:T} is $N[4,3] = 6$.
Numerically we observe each trapping region has a single three-dimensional attractor.
Five of the trapping regions are comprised of $12$ boxes.
The sixth (black in Fig.~\ref{fig:zccqq3d}) is comprised of four boxes
and corresponds to the orbit
\begin{equation}
\begin{split}
v &= (2,1,0), \\
\psi(v) &= (1,0,3), \\
\psi^2(v) &= (0,3,2), \\
\psi^3(v) &= (3,2,1), \\
\psi^4(v) &= v.
\end{split}
\nonumber
\end{equation}

To prove Proposition \ref{pr:numberOrbits} we first establish three lemmas.

\begin{lemma}
If $i$ is not a multiple of $k$ then $\psi^i$ has no fixed points.
\label{le:periodIsMultipleOfk}
\end{lemma}

\begin{proof}
Each time we iterate a vector under $\psi$,
the sum (modulo $k$) of the components of the vector increases by $1$.
If $i$ is not a multiple of $k$,
the sum of the components of $v$ is not equal to the sum of the components of $\psi^i(v)$,
so certainly $v \ne \psi^i(v)$.
\end{proof}

\begin{lemma}
Given $u \in \mathbb{Z}_k^t$ ($t \ge 1$) define a map
$\phi : \mathbb{Z}_k^t \to \mathbb{Z}_k^t$ by
\begin{equation}
\phi(w) = \left( w_2 + u_1, w_3 + u_2, \ldots, w_t + u_{t-1}, w_1 + u_t \right).
\label{eq:phi}
\end{equation}
Let $S = \sum_{j=1}^t u_j \modNormal k$.
Then the fixed point equation
\begin{equation}
w = \phi(w),
\label{eq:generalFixedPointEq}
\end{equation}
has $k$ solutions if $S = 0$ and no solutions otherwise.
\label{le:kSolns}
\end{lemma}

\begin{proof}
By eliminating $w_2,\ldots,w_t$, the
$t$-dimensional fixed point equation \eqref{eq:generalFixedPointEq}
reduces to the scalar equation $w_1 = w_1 + S$, with all other components of $w$ determined uniquely from $w_1$.
If $S \ne 0$, \eqref{eq:generalFixedPointEq} has no solutions;
if $S = 0$, $w_1$ can take any value in $\mathbb{Z}_k$,
and this generates all solutions, so the number of solutions is $k$.
\end{proof}

\begin{lemma}
Let $k$, $n$, and $a$ be as in Theorem \ref{th:nd}, and $b = \frac{n}{a}$.
An integer $j \ge 1$ is a multiple of ${\rm gcd}(j k,n)$ if and only if $j$ is a multiple of $b$.
\label{le:multiples}
\end{lemma}

\begin{proof}
First suppose $j$ is not a multiple of $b$.
Then the prime factorisation of $j$ lacks a power of a prime $p$
that is present the prime factorisation of $b$.
Moreover, $p|k$ by the definition of $a$.
Thus $j k$ and $n$ both contain more powers of $p$ than $j$,
so $j$ is not a multiple of ${\rm gcd}(j k,n)$.

To verify the converse we multiply ${\rm gcd}(k,a) = 1$ by $b$ to obtain ${\rm gcd}(b k,n) = b$.
Then ${\rm gcd}(i b k,n) = {\rm gcd}(i b,n)$ for any positive integer $i$.
So $i b$ is a multiple of ${\rm gcd}(i b k,n)$ as required.
\end{proof}

\begin{proof}[Proof of Proposition \ref{pr:numberOrbits}]
The $n^{\rm th}$ iterate of $\psi$ is
$\psi^n(v) = \left( v_1 + 1, v_2 + 1, \ldots, v_n + 1 \right)$.
Thus $\psi^{k n} = \id$ (the identity map)
and $\psi^i \ne \id$ for all $1 \le i < k n$.
Thus the set
\begin{equation}
G = \left\{ \id, \psi, \psi^2, \ldots, \psi^{k n - 1} \right\},
\nonumber
\end{equation}
together with the composition operator, is a group acting on $\mathbb{Z}_k^n$.
In the context of $G$,
the {\em orbit} of any $v \in \mathbb{Z}_k^n$
is the set $\{ \tau(v) \,|\, \tau \in G \}$.
These orbits are equivalent to orbits of $\psi$, thus $N$ (the number of orbits of $\psi$) is equal to the number of orbits of $G$.

Burnside's lemma \cite{Ga98,VaWi01} gives
\begin{equation}
N = \frac{1}{|G|} \sum_{\tau \in G} \cF[\tau],
\nonumber
\end{equation}
where $|G|$ is the number of elements in $G$
and $\cF[\tau]$ is the number of fixed points of $\tau$
(i.e.~the number of vectors $v \in \mathbb{Z}_k^n$ for which $\tau(v) = v$).
So
\begin{equation}
N = \frac{1}{k n} \sum_{i=1}^{k n} \cF[\psi^i].
\nonumber
\end{equation}
By Lemma \ref{le:periodIsMultipleOfk} this reduces to
\begin{equation}
N = \frac{1}{k n} \sum_{j=1}^n \cF[\psi^{j k}].
\label{eq:Burnside3}
\end{equation}
In \eqref{eq:Burnside3} we can uniquely write
\begin{equation}
j k = q n + r,
\label{eq:jk}
\end{equation}
for $0 \le q \le k$ and $0 \le r < n-1$.
By the definition of $\psi$,
\begin{equation}
\psi^{j k}(v) = \left( v_{r+1} + q, v_{r+2} + q, \ldots, v_n + q, v_1 + q + 1, v_2 + q + 1, \ldots, v_r + q + 1 \right).
\label{eq:psijk}
\end{equation}
Now write $n = s t$ where $s = {\rm gcd}(r,n)$.
By \eqref{eq:psijk} the fixed point equation
$\psi^{j k}(v) = v$ decouples into $s$ identical $t$-dimensional systems of the form \eqref{eq:generalFixedPointEq}.
Thus Lemma \ref{le:kSolns} implies
\begin{equation}
\cF[\psi^{j k}] = \begin{cases}
k^s, & \text{if $S = 0$}, \\
0, & \text{otherwise},
\end{cases}
\label{eq:numberOfSolutions}
\end{equation}
where $S$ is the sum of the components of $u$ in \eqref{eq:phi}.
Observe $S = \frac{j k}{s} \modNormal k$,
because the sum of the constants on the right hand-side of \eqref{eq:psijk} is $j k$,
and we have $s$ identical instances of \eqref{eq:generalFixedPointEq}.
Thus $S = 0$ if and only if $j$ is a multiple of $s$.
Since $s = {\rm gcd}(j k,n)$, Lemma \ref{le:multiples}
implies $S = 0$ if and only if $j = i b$, for some $i \in \{ 1,2,\ldots,a \}$.
In this case
\begin{equation}
s = {\rm gcd}(i b k,n) = {\rm gcd}(i b,n) = b \,{\rm gcd}(i,a).
\nonumber
\end{equation}
Using \eqref{eq:numberOfSolutions} we can therefore write \eqref{eq:Burnside3} as
\begin{equation}
N = \frac{1}{k n} \sum_{i=1}^a k^{b \,{\rm gcd}(i,a)}.
\nonumber
\end{equation}
Each ${\rm gcd}(i,a)$ is a divisor of $a$.
By the definition of Euler's totient function,
for any divisor $d$ the number of values of $i \in \{ 1,2,\ldots,a \}$
for which ${\rm gcd}(i,a) = d$ is $\varphi \left( \frac{a}{d} \right)$.
Thus
\begin{equation}
N = \frac{1}{k n} \sum_{d|a} \varphi \left( \frac{a}{d} \right) k^{b d},
\nonumber
\end{equation}
and by replacing $d$ with $\frac{a}{d}$ we obtain \eqref{eq:Nformula}.
\end{proof}

\section{Expanding dynamics}
\label{sec:expansion}

Above we observed that for any set $T_v$ of the form \eqref{eq:T},
the scaled set $\mu T_v$ is a trapping region for any map $f$
that satisfies the conditions of Lemma \ref{le:boxImagePerturbed}.
In this section we show that $f^{kn}$ is piecewise-$C^r$ and expanding on $\mu T_v$.

\begin{lemma}
Let $(a_L,a_R) \in S_k$ ($k \ge 2$),
let $\delta > 0$ be as in Lemma \ref{le:Jimap},
and let $U_3$ and $\mu_3$ be as in Lemma \ref{le:boxImagePerturbed}.
There exists a neighbourhood $U_4 \subset U_3$ of $(d^L,d^R)$ such that
for any piecewise-$C^r$ ($r \ge 1$) map $f$ of the form \eqref{eq:ndbcnfWithHOTs} with $(c^L,c^R) \in U_4$
there exists $\mu_4 \in (0,\mu_3]$ such that $f^{k n}$ is piecewise-$C^r$ and expanding on $\mu T_v$
for all $\mu \in (0,\mu_4)$ and all $v \in \mathbb{Z}_k^n$.
\label{le:expanding}
\end{lemma}

\begin{proof}
Let $y \in \mathbb{R}^n$ be such that $g^j(y)_1 \ne 0$ for all $j \ge 0$.
Then $\left( \rD g^j \right)(y)$ is defined for all $j \ge 1$.
By \eqref{eq:simpleForm} and \eqref{eq:skewTentMap},
\begin{equation}
(\rD g)(y) = \begin{bmatrix}
0 & 1 \\
0 && \ddots \\
\vdots &&& 1 \\
h'(y_1)
\end{bmatrix},
\label{eq:exp1}
\end{equation}
where $h'(y_1) = a_L$ if $y_1 < 0$ and $h'(y_1) = a_R$ if $y_1 > 0$.
For all $j = 2,3,\ldots,n$, the first component of $g^{j-1}(y)$ is $y_j + 1$, thus
\begin{equation}
(\rD g) \left( g^{j-1}(y) \right) = \begin{bmatrix}
0 & 1 \\
0 && \ddots \\
\vdots &&& 1 \\
h'(y_j + 1)
\end{bmatrix}.
\label{eq:exp2}
\end{equation}
By multiplying together the $n$ matrices \eqref{eq:exp1} and \eqref{eq:exp2} for $j = 2,3,\ldots,n$ we obtain
\begin{equation}
\left( \rD g^n \right)(y)
= \begin{bmatrix}
h'(y_1) &&& \\
& h'(y_2+1) && \\
&& \ddots & \\
&&& h'(y_n+1)
\end{bmatrix},
\label{eq:exp3}
\end{equation}
which is diagonal.
Using Lemma \ref{le:directProduct} we take a product of $k$ instances of \eqref{eq:exp3} to obtain
\begin{equation}
\left( \rD g^{k n} \right)(y)
= \begin{bmatrix}
\prod_{i=0}^{k-1} h' \left( h^i(x_1) \right) &&& \\
& \prod_{i=0}^{k-1} h' \left( h^i(x_2+1) \right) && \\
&& \ddots & \\
&&& \prod_{i=0}^{k-1} h' \left( h^i(x_n+1) \right)
\end{bmatrix}.
\label{eq:exp4}
\end{equation}
But on $\bigcup_i J_i$, orbits of $g$ cycle through the intervals in order,
thus the diagonal entries of \eqref{eq:exp4} can only take values in $\left\{ a_L^{k-2} a_R^2, a_L^{k-1} a_R \right\}$.
Notice $a_L^{k-2} a_R^2 > a_L^{k-1} |a_R| > 1$ for any $(a_L,a_R) \in S_k$.
Thus each smooth component of $g^{k n} \big|_{\tilde{\Phi}}$
has a Jacobian matrix that is diagonal with each diagonal entry greater than one in absolute value.
Thus $g^{k n}$ is piecewise-linear and expanding.

By perturbing $g$ to $f \in U_3$ and considering $\mu \in (0,\mu_3)$,
no additional symbolic itineraries are possible for orbits in $\mu T_v$.
Moreover, by Lemma \ref{le:perturb} the Jacobian matrix of each smooth piece of $f^{k n}$
is a small perturbation of a diagonal matrix with diagonal entries greater than one in absolute value.
Hence there exists a neighbourhood $U_4 \subset U_3$ of $(d^L,d^R)$ and $\mu_4 \in (0,\mu_3]$
such that for any $f \in U_4$ and $\mu \in (0,\mu_4)$,
each smooth piece of $f^{k n}$ is expanding.
Further each smooth component is $C^r$ because $f$ is piecewise-$C^r$.

Finally we note that the regions on which $f^{k n}$ is smooth are nice: in accordance with Definition \ref{df:pws}.
The boundaries of these regions are defined implicitly by $f^i(x)_1 = 0$, for $i = 0,1,\ldots,k n - 1$.
For the map $g$, for any such value of $i$ we have $g^i(x)_1 = \alpha x_j + \beta$
for some $\alpha \ne 0$, $\beta \in \mathbb{R}$, and $j \in \{ 1,2,\ldots,n \}$.
That is, each boundary is a hyperplane normal to one of the coordinate axes
(in fact they divide each $\Phi_v$ into $2^n$ regions).
Since $\alpha \ne 0$, these boundaries perturb smoothly with no additional intersections,
so $f^{k n} \big|_{\mu T_v}$ is indeed piecewise-$C^r$ and expanding
assuming $U_4$ and $\mu_4$ are sufficiently small.
\end{proof}

\section{Collating the results to prove Theorems \ref{th:2d} and \ref{th:nd}}
\label{sec:proof}

\begin{proof}[Proof of Theorem \ref{th:nd}]
Let $U_2$ and $\mu_2$ be as in Lemma \ref{le:stableFixedPoint},
and $U_4$ and $\mu_4$ be as in Lemma \ref{le:expanding}.
Let $U = U_2 \cap U_4$ be a neighbourhood of $(d^L,d^R)$, and $\mu_0 = {\rm min}(\mu_2,\mu_4) > 0$.

Choose any $(c^L,c^R) \in U$ and $\mu \in (-\mu_0,\mu_0)$.
If $\mu < 0$ then $f$ has an asymptotically stable fixed point by Lemma \ref{le:stableFixedPoint}.
Now suppose $\mu > 0$.
For any $v \in \mathbb{Z}_k^n$, let $T_v$ be given by \eqref{eq:T}
using $\delta > 0$ as in Lemma \ref{le:Jimap}.
Then by Lemma \ref{le:boxImagePerturbed}, the scaled set $\mu T_v$ is a trapping region for $f$.
By Lemma \ref{le:expanding}, $f^{k n}$ is piecewise-$C^r$ and expanding on $\mu T_v$.
Finally by Proposition \ref{pr:numberOrbits} the number of such trapping regions is
given by \eqref{eq:Nformula}.
\end{proof}

\begin{proof}[Proof of Theorem \ref{th:2d}]
This follows from Theorem \ref{th:nd} with $n=2$ and $k = 2 N$ (and results in $m = k n$).
\end{proof}

\section{Discussion}
\label{sec:conc}

In this paper we have extended the ideas of \cite{PuRo18,Gl15b,WoYa19}
to show that stable fixed points can bifurcate to
any number of coexisting chaotic attractors in BCBs.
In fact the dynamics on the attractors is expanding
so the attractors are typically all $n$-dimensional.

The example in Fig.~\ref{fig:exBifDiag}
was obtained by first choosing a point in $S_3$, see Fig.~\ref{fig:zccbifSet1dBCNF}.
Specifically we used $(a_L,a_R) = (0.62,-3)$.
As a small perturbation from $d^L = \begin{bmatrix} 0 \\ a_L \end{bmatrix}$ and $d^R = \begin{bmatrix} 0 \\ a_R \end{bmatrix}$,
we used $c^L = \begin{bmatrix} \tau_L \\ -\delta_L \end{bmatrix}$ and $c^L = \begin{bmatrix} \tau_R \\ -\delta_R \end{bmatrix}$
with the values \eqref{eq:xiEx}.
We also incorporated nonlinearity through $E_L$ \eqref{eq:ELEREx}
to illustrate what is likely a typical breakdown of coexistence at $\mu \approx 0.01$.

The corresponding simple form $g$ has $N[3,2] = 2$ trapping regions of the form $T_v$ \eqref{eq:T}.
These are for $v = (0,0)$ (comprised of six boxes) and $v(1,0)$ (comprised of three boxes).
This division is plainly evident in the geometry of the attractors shown in Fig.~\ref{fig:exBifDiag}-b.

In general for the map $g$
each trapping region $T_v$ contains a unique attractor.
This is because on $I_0 \cup \cdots \cup I_{k-1}$ the corresponding skew tent map $h$ is {\em locally eventually onto} \cite{Ru17}
(this follows from the proof of Lemma 2.2 of \cite{VeGl90}).
Consequently the direct product $g^n$ is locally eventually onto on a Cartesian product of unions of intervals
(this idea is used in the proof of Lemma 2.2 of \cite{PuRo18}).
It follows $g$ is transitive on $\bigcap_{i \ge 0} g^i(T_v)$, so $T_v$ contains a unique attractor.
It remains to determine whether or not this generalises from $g$ to all sufficiently small perturbations $f$.

Instead of perturbing about $c^L = a_L e_n$ and $c^R = a_R e_n$, for carefully chosen values of $a_L$ and $a_R$,
we could instead perturb about $c^L = a_L e_j$ and $c^R = a_R e_j$, where $j < n$.
This allows $f$ to be invertible and in this setting we expect to see $j$-dimensional attractors,
although the expansion arguments of \S\ref{sec:expansion} cannot be easily generalised to establish this.
The case $j=1$ and $n=2$ with the Lozi map (a subfamily of the two-dimensional border-collision normal form)
was considered by Cao and Liu in \cite{CaLi98}.
They showed that attractors of sufficiently small perturbations
exhibit chaos in the sense of Devaney \cite{De89}.
From an ergodic viewpoint this type of problem was studied in \cite{WaYo01,Yo85}.

\section*{Acknowledgements}

This work was supported by Marsden Fund contract MAU1809, managed by Royal Society Te Ap\={a}rangi.
The author thanks Paul Glendinning and Chris Tuffley for discussions that helped improve the results.

\appendix

\end{document}